\documentclass{amsart}
\usepackage{amsmath}
\usepackage{graphicx}
\usepackage{multicolumn}
\usepackage{bm,color}

\title
[Unbounded generalization ...]
{Unbounded generalization of \\
the Baker-Campbell-Hausdorff formulae}

\author{Yoritaka Iwata \vspace{6mm}}

\address[Y. Iwata]{Kansai University, Japan}
\curraddr{Kansai University, Yamate-cho 3-3-35, Osaka 564-8680, Japan}
\email{iwata$\_$phys@08.alumni.u-tokyo.ac.jp \\
{\it ORCID}: 0000-0002-6898-7505 }

\thanks{The author is grateful to Prof. Emeritus Dr. Hiroki Tanabe for fruitful comments.}

\theoremstyle{plain}
\newtheorem{theorem}{Theorem}

\newtheorem{lemma}[theorem]{Lemma}

\begin{document}

\begin{abstract}
Based on the operator representation on the module over Banach algebra, the Campbell-Baker-Hausdorff formula is generalized to be applicable to unbounded infinitesimal generators in Banach spaces.
For this purpose, it is necessary to extract a bounded part from a given generally-unbounded infinitesimal generator. In this paper, much attention is paid to clarify the unknown relation between the commutator product and the logarithm of operators.
In conclusion, by means of the logarithmic representation of generally-unbounded operators, the Campbell-Baker-Hausdorff formula is generalized.
As an application, a general form of von Neumann equation is presented in which the role of logarithm as the commutator is clearly seen. 
\end{abstract}

\maketitle


\section{Introduction}
Let us assume a finite/infinite dimensional Banach space, and $[\cdot, \cdot]$ denote the commutator on the Banach space.
For bounded operators $A$ and $B$ on the Banach space, the Campbell-Baker-Hausdorff formula in the Banach space reads
\begin{equation} \begin{array}{ll}   \label{master}
 e^{A} e^{B} = \exp \left[ A+ B + \frac{1}{2} [A,B] + \frac{1}{12}[A,[A,B]]- \frac{1}{12}[B,[A,B]] + \cdots \right],
 \end{array} \end{equation}
that is proved by means of the asymptotic expansion in the theory of Lie algebra (for a textbook, see \cite{73sagel}).
Here operators $A$ and $B$ are generally assumed to be non-commutative
\[ [A, B] = AB - BA \neq 0, \]
and the formula results in $ e^{A} e^{B} = e^{A+ B}$, if $A$ and $B$ commute. 

The generator of the right hand side is a function of $A$ and $B$:
\[ \begin{array}{ll}
A+ B + \frac{1}{2} [A,B] + \frac{1}{12}[A,[A,B]]- \frac{1}{12}[B,[A,B]] + \cdots
 \end{array}
\]
which consists of commutators of operators.
The exponential functions of operator, which are represented by the form of convergent power series, are trivially well defined only when $A$ and $B$ are bounded on Banach spaces.
The discrepancy between the left and right hand sides appears in unbounded cases.
Indeed, as the commutator products of operators are found only in the right hand side, the narrower domain space of infinitesimal generator seems to be true for the right hand side.
It is remarkable for the left hand side that $e^A$ and $e^B$ are definitely bounded on Banach spaces, as far as they are generated based on the standard theory of evolution equations ~\cite{61kato,70kato,73kato,60tanabe,61tanabe,79tanabe} (cf. Hille-Yosida theorem in which the exponential functions of generally-unbounded operators are defined by the limit of Yosida approximation).

In this paper, by means of the logarithmic representation of infinitesimal generator (more precisely, alternative infinitesimal generators \cite{17iwata-3} defined by the logarithmic representation of operators \cite{17iwata-1,21iwata}), the Campbell-Baker-Hausdorff formula is generalized to unbounded cases, in which infinitesimal generators are assumed to be general $t$-dependent unbounded infinitesimal generators in Banach spaces.

\section{Background, mathematical setting, and basic concept}
\subsection{Background}
There are several types of the Campbell-Baker-Hausdorff formulae.
Here the following two types are mainly studied.
\begin{equation} \label{cbh2} \begin{array}{ll}
e^{A} B e^{-A} =  B + [A,B] + \frac{1}{2}[A,[A,B]]+ \cdots, \vspace{3.5mm} \\
e^{A} e^{\pm B} = \exp \left[ A \pm B \pm \frac{1}{2} [A,  B] \pm \frac{1}{12}[A,[A, B]] + \frac{1}{12}[ B,[A, B]] + \cdots \right],
 \end{array} \end{equation}
where the first one is referred to the Type I and the second one to the Type II in this paper. 
The first equation of (\ref{cbh2}) is regarded as a representation formula based on the similarity transform.
The second equation of (\ref{cbh2}) shows an exponential function of non-commutative operators.
Both of them show the relation between the non-commutative property of operator.
The exponential functions of those operators are trivially true only when $A$ and $B$ are bounded on a Banach space.
Operators $A$ and $B$ must be bounded on a given Banach space for the validity of these formulae, so that it is necessary to introduce additional treatment to be applicable to differential operators.
Such a limitation is removed in this paper.

\subsection{Mathematical setting}
Here a general and abstract framework is prepared.
Let $X$ be a finite/infinite dimensional Banach space and $B(X)$ be a set of bounded operators on $X$.
The norms of both $X$ and $B(X)$ are denoted by $\| \cdot \|$ if there is no ambiguity.
Let $t$ and $s$ be real numbers included in a finite interval $[0, T]$, and $U(t,s)$ be an evolution operator in $X$. 
The two parameter semigroups $U_i (t,s)$, which are continuous with respect to both parameters $t$ and $s$, are assumed to be a bounded operator on $X$ ($i$=1,2).
That is, the boundedness condition 
\begin{equation} \label{bound} \begin{array}{ll}
\| U_i(t,s) \| \le M_i e^{\omega_i (t-s)}
\end{array} \end{equation}
is assumed for certain real numbers $M_i$ and $\omega_i$.
Following the standard theory of abstract evolution equations, the semigroup property:
\[ \begin{array}{ll}
 U_i(t,r) U_i(r,s) = U_i(t,s)  
\end{array} \]
is assumed to be satisfied for arbitrary $s \le r \le t$ included in a finite interval $[0, T]$.
Let evolution operators $U_i(t,s)$ be generated by $A_i(t): D(A_i(t)) \to X$ respectively.
Then, for an unknown function $u(t) \in C^1([0,T];X)$,
\begin{equation} \label{1stmastereq}
\partial_t u(t) = A_i (t) u(t)
\end{equation}
is satisfied in $X$.
Since $A_i(t)$ are generally unbounded in $X$, $D(A_i(t))$ are strictly included in $X$.
Consequently the operator $A_i(t)$ are infinitesimal generators of $C_0$-semigroup $U(t,s)$ in which operators $A_i(t)$ satisfy the assumption of generation theorem of two parameter $C_0$-semigroup (for one parameter case, see Hille-Yosida theorem).
It is remarkable that evolution operators $U_i(t,s)$, which correspond to the exponential functions of operator, cannot be represented in general by the convergent power series in unbounded cases. 
Note that $A_i$ defined in $X$ are general enough to be applicable to the matrix situation.

\subsection{Logarithmic representation of operators}
As a basic concept, the logarithmic representation of operators \cite{17iwata-1,20iwata-2} is briefly reviewed.
Using the Riesz-Dunford integral \cite{43dunford}, the infinitesimal generator $A(t)$ of the first-order evolution equation is written by
\begin{equation} \label{mastart}
A(t)   
=   (I + \kappa U(s,t))
\partial_t  {\rm Log} ( U(t,s) + \kappa I)
\end{equation}
under the commutation between $\partial_t U(t,s)$ and $U(t,s)$~\cite{17iwata-1}, where ${\rm Log}$ means the principal branch of logarithm, and $U(t,s)$ is temporarily assumed to be a group (i.e., existence of $U(s,t) = U(t,s)^{-1}$ is temporarily assumed to be valid for any $0 \le s \le t \le T$), and not only a semigroup. 
Note that $\kappa$ is a certain complex number chosen to be included in the resolvent set of $U(t,s)$, and $\partial_t$ denotes a certain kind of weak differential~\cite{20iwata-4}.
The validity of \eqref{mastart} is confirmed formally by 
\[ \begin{array}{ll}
(I + \kappa U(s, t) ) \partial_t {\rm Log}(U(t, s) +  \kappa I)   \vspace{1.5mm}\\
= U(s, t)(U(t, s) + \kappa I)\partial_t U(t, s)(U(t, s) + \kappa I)^{-1}  \vspace{1.5mm}\\
= U(s, t) \partial_t U(t, s)   \vspace{1.5mm}\\
 = U(s, t)A(t)U(t, s) = A(t)
\end{array} \]
because $\partial_t U(t,s) = A(t) U(t,s)$ commutes with $U(t,s)$, and therefore with $U(s,t)$.
This relation is associated with the abstract form \cite{19iwata-2} of the Cole-Hopf transform \cite{50hopf,51cole}, and further associated with the abstract form \cite{20iwata-1} of the Miura transform~\cite{68miura}.
The correspondence between $\partial_t  \log u$ and $A(t)$ can be understood by $U(s, t) \partial_t U(t, s)   = A(t)$ shown above.
Indeed, for $u(t) \in ({\bm R}; X)$ satisfying $\| u \| > 0$,
\[
 u^{-1}   \partial_t  u   =   \partial_t  \log u 
 \quad  \Rightarrow \quad  
  \partial_t  u   =  ( \partial_t  \log u) u 
\]
is valid under the commutation between $\partial_t u$ and $u$.
Consequently, it is understood that $A(t)$ in Eq.~\eqref{mastart} is a mathematically rigorous representation for a formal representation ${\rm Log} U(t,s)$.

By introducing alternative infinitesimal generator $a(t,s)$ \cite{17iwata-3} satisfying
\[
e^{a(t,s)} := U(t,s)+ \kappa I, 
\]
a generalized version of the logarithmic representation
\begin{equation}  \label{altrep}
 A(t)  
=   ( I   - \kappa   e^{-a(t,s)}  )^{-1}   \partial_t  a(t,s)
\end{equation}
is obtained, where $U(t,s)$ is assumed to be only a semigroup. 
The regularized evolution operator $e^{a(t,s)}$ is always well defined by a convergent power series after adjusting the amplitude of $\kappa$, even if the original evolution operator $U(t,s)$ is unbounded \cite{17iwata-3}. 
The right hand side of Eq.~(\ref{altrep})  is actually a generalization of \eqref{mastart}; indeed, by only assuming $U(t,s)$ as a semigroup defined on $X$,  $e^{-a(t,s)}$ is always well defined by a convergent power series, and there is no need to have a temporal assumption for the existence of $U(s,t)= U(t,s)^{-1}$. 
It is remarkable here that $e^{-a(t,s)} = e^{a(s,t)}$ is not necessarily satisfied~\cite{17iwata-3}.
The validity of Eq.~\eqref{altrep} is briefly seen in the following.
Using the generalized representation,
\[ \begin{array}{ll}
e^{a(t,s)} = U(t, s) + \kappa I \quad \Rightarrow \quad a(t, s) = {\rm Log}(U(t, s) + \kappa I) \vspace{1.5mm} \\
\qquad  \Rightarrow \quad \partial_t a(t, s) = \left[ \partial_t U(t, s) \right]  (U(t, s) + \kappa I)^{-1} 
= A(t)U(t, s)(U(t, s) + \kappa I)^{-1}
\end{array} \]
is obtained.
It leads to
\[ \begin{array}{ll}
(I - \kappa e^{-a(t,s)})^{-1} \partial_t a(t,s)  =
(I - \kappa e^{-a(t,s)})^{-1} A(t)U(t, s)(U(t, s) + \kappa I)^{-1},
\end{array} \]
and therefore
\[ \begin{array}{ll}
A (t) =   (U(t,s) + \kappa I )  U(t,s) ^{-1}  A(t)U(t, s)(U(t, s) + \kappa I)^{-1}   \vspace{1.5mm} \\
\quad \Leftrightarrow \quad
A (t) = (I - \kappa e^{-a(t,s)})^{-1} A(t) (I - \kappa  e^{-a(t,s)} )
\end{array} \]
is valid under the commutation assumption.
It simply shows the consistency of representations using the alternative infinitesimal generator $a(t,s)$.
In the following, the logarithmic representation \eqref{altrep} is definitely used, and the original representation \eqref{mastart} appears only if it is necessary. 
Based on ordinary and generalized logarithmic representations, a set of generally unbounded infinitesimal generators has been shown to hold the structure of the module over Banach algebra \cite{20iwata-2}.

\section{Results}
\subsection{Bounded situation}
Here the standard proofs are recalled for bounded $A$ and $B$ on a Banach space $X$.
The first equation of (\ref{cbh2}) is proved in bounded cases.

\begin{lemma} \label{lem1}
For bounded operators $A$ and $B$ on $X$,
\[ e^{tA} B e^{-tA} = B +t [A,B] + \frac{1}{2} t^2[A,[A,B]] + o(t^3), \]
is valid in $X$, and therefore
\[ e^{A} B e^{-A} = B + [A,B] + \frac{1}{2}[A,[A,B]] + \cdots \]
is true by taking $t=1$, where $e^{\pm A}$ are defined by the convergent power series.
\end{lemma}

\begin{proof}
Let us take a function $e^{tA} B e^{-tA} $.
The first derivative is 
\[ \begin{array}{ll}
 \frac{d}{dt} \left( e^{tA} B e^{-tA} \right)
 = \frac{d}{dt} \left( e^{tA} \right) B e^{-tA}  +  e^{tA} \frac{d}{dt} \left(  B e^{-tA} \right)  \vspace{1.5mm} \\
 = e^{tA}  \left( AB  \right) e^{-tA}  +  e^{tA}  (-BA)   e^{-tA}  \vspace{1.5mm} \\
 = e^{tA}  \left( AB - BA  \right) e^{-tA}  
 = e^{tA}  \left[ A,B \right] e^{-tA} , 
\end{array} \]
and the second derivative is
\[ \begin{array}{ll}
 \frac{d^2}{dt^2} \left( e^{tA} B e^{-tA} \right)
 = \frac{d}{dt} \left(  e^{tA}  \left[ A,B \right] e^{-tA}    \right)   \vspace{1.5mm} \\
 = e^{tA}  \left( A[A,B] - [A,B]A  \right) e^{-tA}   \vspace{1.5mm} \\
 = e^{tA}  \left[ A, [A,B]  \right]  e^{-tA}  .
\end{array} \]
Mathematical induction leads to 
\[ \frac{d^n}{dt^n}(e^{tA}Be^{-tA})=e^{tA}\overbrace{[A,[A,\cdots,[A}^{n-1},[A,B]\overbrace{]]\cdots]}^{n-1}e^{-tA}. \]
The formula is obtained by the Taylor expansion
\[ e^{tA} B e^{-tA} = B +t [A,B] + \frac{1}{2} t^2[A,[A,B]] + o(t^3), \]
and eventually by taking $t=1$.
Note that, as seen in the convergence radius arguments in infinite series, the asymptotic expansion itself cannot be valid for unbounded $A$ and $B$ in $X$.
\end{proof}

Next the second equation of (\ref{cbh2}) is proved in bounded cases.
Although the Baker-Campbell-Hausdorff formula is known to be not necessarily convergent in some concrete problems, this issue is cured by limiting the property of semigroup (for example contraction semigroup or unitary group satisfy the condition; cf. condition \eqref{cond1}) in the following Lemma.

\begin{lemma} \label{lem2}
Let $\delta$ be a certain real number satisfying $0< \delta \le  \sqrt{2}$.
For bounded operators $A$ and $B$ on $X$, let $\left\| e^{tA} \right\|$ and $\left\| e^{tB} \right\|$ be sufficiently small satisfying
\begin{equation} \label{cond1} 
\left\| e^{tA} \right\| < \delta, \qquad  \left\| e^{tB} \right\| < \delta 
\end{equation} 
for a time interval $t \in [0,1].$
Then the product of $e^{A}$ and  $e^{\pm B}$ is represented by
\[ \begin{array}{ll}
 e^{tA} e^{tB}  =  I +  (A+B)t + \frac{1}{2} \left( (A + B)^2   +  [A,B]  \right)t^2 + o(t^3),
\end{array} \]
in $X$, and therefore by
\[ e^{A} e^{\pm B} = \exp \left[ A \pm B \pm \frac{1}{2} [A,B] \pm \frac{1}{12}[A,[A,B]] + \frac{1}{12}[B,[A,B]] + \cdots \right] \]
if $t$ is equal to 1.
\end{lemma}

\begin{proof}
By introducing a varible $t$ again, and it is sufficient to prove the formula
\[ e^{A} e^{B} = \exp \left[ A + B + \frac{1}{2} [A,B] + \frac{1}{12}[A,[A,B]] + \frac{1}{12}[B,[A,B]] + \cdots \right] \]
is proved.
The first derivative is 
\[ \begin{array}{ll}
 \frac{d}{dt} \left( e^{tA} e^{tB} \right)
 = e^{tA} A e^{tB} +  e^{tA} B e^{tB} 
 = e^{tA} (A + B) e^{tB}, 
\end{array} \] 
and the second derivative is
\[ \begin{array}{ll}
 \frac{d^2}{dt^2} \left( e^{tA} e^{tB} \right) =
 \frac{d}{dt} \left(  e^{tA} (A + B) e^{tB}  \right) =
 e^{tA} A (A + B) e^{tB}   +  e^{tA} (A + B) B  e^{tB}   \vspace{1.5mm} \\
=  e^{tA} \left( A (A + B)  +  (A + B) B \right)  e^{tB}   \vspace{1.5mm} \\
=  e^{tA} \left( A^2 + A B  +  AB + B^2  + BA - BA  \right)  e^{tB}   \vspace{1.5mm} \\
=  e^{tA} \left( A^2 + A B  +  BA + B^2  + (AB - BA)  \right)  e^{tB}   \vspace{1.5mm} \\
=  e^{tA} \left( (A + B)^2   +  [A,B]  \right)  e^{tB} .  \vspace{1.5mm} \\
\end{array} \]
It results in
\[ \begin{array}{ll}
 e^{tA} e^{tB} = I + (A+B)t + \frac{1}{2} \left( (A + B)^2   +  [A,B]  \right)t^2 + o(t^3),
\end{array} \]
and therefore in
\[ \begin{array}{ll}
 e^{A} e^{B} = I + (A+B) + \frac{1}{2} \left( (A + B)^2   +  [A,B]  \right) + \cdots
\end{array} \]
is obtained.
Note again that, as seen in the convergence radius arguments in infinite series, the asymptotic expansion itself cannot be valid for unbounded $A$ and $B$ in $X$.

Next, its logarithm is written by
\[ \begin{array}{ll}
{\rm Log} \left( e^{tA} e^{tB} \right)
 = {\rm Log} \left(  I + (A+B)t + \frac{1}{2} \left( (A + B)^2   +  [A,B]  \right)t^2 + \cdots \right)    \vspace{1.5mm} \\
 =   (A+B)t + \frac{1}{2} \left( (A + B)^2   +  [A,B]  \right)t^2 + \cdots)   \vspace{0.5mm} \\
\qquad  - \frac{1}{2} \left\{  (A+B)t + \frac{1}{2} \left( (A + B)^2   +  [A,B]  \right)t^2 + \cdots) \right\} ^2  \vspace{0.5mm} \\
\qquad  + \frac{1}{3} \left\{  (A+B)t + \frac{1}{2} \left( (A + B)^2   +  [A,B]  \right)t^2 + \cdots) \right\} ^3   \vspace{1.5mm} \\
 =   (A+B)t + \frac{1}{2}   [A,B]    t^2   + \cdots
\end{array} \]
if $ \left\| (A+B)t + \frac{1}{2} \left( (A + B)^2   +  [A,B]  \right)t^2 + \cdots \right\| < 1 $ is satisfied.
The condition arises from the convergence radius of the logarithm. 
Consequently, 
\[ \begin{array}{ll}
e^{tA} e^{tB} 
 = \exp \left[ (A+B)t + \frac{1}{2}   [A,B]    t^2   + \cdots \right]
\end{array} \]
is obtained, and it is sufficient to take $t=1$.
Here note again that since the validity of
\[ \begin{array}{ll}
 e^{tA} e^{tB}  - I  =  (A+B)t + \frac{1}{2} \left( (A + B)^2   +  [A,B]  \right)t^2 + o(t^3),
\end{array} \]
the condition is replaced by
\[ \left\| e^{tA} e^{tB} \right\| < 2 \] 
so that it is sufficient to assume, for a certain $\delta$ satisfying $0 < \delta \le \sqrt{2}$,
\[ \left\| e^{tA} \right\| < \delta, \qquad  \left\| e^{tB} \right\| < \delta. \] 
The proof is essentially the same for $e^{A} e^{-B}$.
\end{proof}

\subsection{Unbounded situation}
The Baker-Campbell-Hausdorff type formulae are proved for unbounded $A$ and $B$ in a Banach space $X$.
The first equation of (\ref{cbh2}) is proved for unbounded cases.

\begin{theorem}[Generalized Baker-Campbell-Hausdorff formula: Type I]  \label{thm1}
Let $U_1(t,s)$ and $U_2(t,s)$ be two-parameter $C_0$-semigroups of operator (in other words, evolution operators) in a Banach space $X$.
Evolution operators $U_1(t,s)$ and $U_2(t,s)$ are assumed to be generated by generally unbounded operators $A_1(t)$ and $A_2(t)$ in $X$, respectively. 
For the alternative infinitesimal generators $\partial_t a_1(t,s)$ and $\partial_t a_2(t,s)$ for infinitesimal generators $A_1(t)$ and $A_2(t)$ respectively, the formula of exponential function of operators
\[ \begin{array}{ll}
 e^{ a_1(t,s)}  a_2(t,s) e^{- a_1(t,s)}  \vspace{1.5mm}\\
 =  a_2(t,s) + [ a_1(t,s), a_2(t,s)] + \frac{1}{2}[ a_1(t,s),[ a_1(t,s), a_2(t,s)]] + o(a_i(t,s)^4)
\end{array} \]
is valid in $X$.
\end{theorem}

\begin{proof}
Let $U_i(t,s)$ denote evolution operators generated by $A_i(t)$ respectively ($i=1,2$).
The alternative infinitesimal generator is defined by
\[
a_i(t,s) = {\rm Log} (U_i(t,s) + \kappa I),
\]
and the relation
\[
e^{ a_i(t,s)} = U_i(t,s) + \kappa I
\]
is valid by taking the exponential function as a convergent power series for both sides.
Indeed ${\rm Log} (U_i(t,s) + \kappa I),$ and therefore $a_i(t,s)$ are bounded on $X$.
In exactly the same setting, $e^{-a_i(t,s)}$ is well defined by taking the exponential function as a convergent power series of
\[
- a_i(t,s) = - {\rm Log} (U_i(t,s) + \kappa I).
\]
The formula follows directly from Lemma \ref{lem1}. 
\end{proof}

\begin{theorem}  \label{thm2}
Let $U_1(t,s)$ and $U_2(t,s)$ be two-parameter $C_0$-semigroups of operator (in other words, evolution operators) in a Banach space $X$.
The evolution operators $U_1(t,s)$ and $U_2(t,s)$ are assumed to be generated by generally unbounded operators $A_1(t)$ and $A_2(t)$ in $X$, respectively. 
For the alternative infinitesimal generators $\partial_t a_1(t,s)$ and  $\partial_t a_2(t,s)$ for infinitesimal generators $A_1(t)$ and $A_2(t)$ respectively, the formula of exponential function of operators
\[ \begin{array}{ll}
e^{ a_1(t,s) } e^{ a_2(t,s)}   \vspace{1.5mm} \\
\quad =  I +( a_1(t,s) + a_2(t,s))
 + \frac{1}{2} \left\{ \left( a_1(t,s)   +  a_2(t,s) \right)^2   + [a_1(t,s),  a_2(t,s) ] \right\} + o(a_i(t,s)^3) ,
\end{array} \]
and
\[ \begin{array}{ll}
e^{{\hat \alpha}_1(t,s)} e^{{\hat \alpha}_2(t,s)}  
 =    e^{    {\rm Log} (\kappa + 1) + (\kappa +1 )^{-1} ( a_1(t,s) +  a_2(t,s) ) + \frac{1}{2} (\kappa +1 )^{-1}  [ a_1(t,s), a_2(t,s)]  + o(a_i(t,s)^3) } .
\end{array} \]
are valid in $X$, where the alternative infinitesimal generators ${\hat \alpha}_1(t,s)$ and ${\hat \alpha}_2(t,s)$ of $  e^{ a_1(t,s) } e^{  a_2(t,s)} $ are defined by $e^{{\hat \alpha}_1(t,s)} e^{{\hat \alpha}_2(t,s)} 
:=  e^{ a_1(t,s) } e^{  a_2(t,s)} + \kappa I $.
\end{theorem}

\begin{proof}
Let $U_i(t,s)$ be evolution operators generated by $A_i(t)$ respectively.
The alternative infinitesimal generator is defined by
\begin{equation} \label{altn}
a_i(t,s) = {\rm Log} (  U_i(t,s) + \kappa I)
\end{equation}
and the relation
\[
e^{ a_i(t,s)} =  U_i(t,s)  + \kappa I
\]
is valid by taking the exponential function as a convergent power series for both sides.
Indeed ${\rm Log} (U_i(t,s) + \kappa I),$ and therefore $a_i(t,s)$ are bounded on $X$.
\[
e^{ a_i(t,s)} = I + a_i(t,s) + \frac{1}{2!} a_i(t,s)^2  + \frac{1}{3!} a_i(t,s)^3  \cdots
\]
For the definition of alternative infinitesimal generator, it is sufficient for a certain complex number $\kappa$ to hold sufficient large amplitude $|\kappa|$.
Here is a reason why a common $\kappa$ is applied for both $a_1(t,s)$ and  $a_2(t,s)$.
The formula follows from the expansion
\[ \begin{array}{ll}
e^{ a_1(t,s) } e^{  a_2(t,s)}   \vspace{1.5mm} \\
\quad = ( I + a_1(t,s) + \frac{1}{2!} a_1(t,s)^2  +  \cdots) 
  ( I + a_2(t,s) + \frac{1}{2!} a_2(t,s)^2  +   \cdots)    \vspace{1.5mm} \\
\quad =  I +( a_1(t,s) + a_2(t,s))
  + \frac{1}{2} a_1(t,s)^2   + \frac{1}{2}  a_2(t,s)^2   + a_1(t,s)  a_2(t,s)  \cdots    \vspace{1.5mm} \\
\quad =  I +( a_1(t,s) + a_2(t,s))   \vspace{1.5mm} \\
\qquad  + \frac{1}{2} \left\{ \left( a_1(t,s)   +  a_2(t,s) \right)^2 -  a_2(t,s)  a_1(t,s) +  a_1(t,s)  a_2(t,s)   \right\}  \cdots    \vspace{1.5mm} \\
\quad =  I +( a_1(t,s) + a_2(t,s))
 + \frac{1}{2} \left\{  \left( a_1(t,s)   +  a_2(t,s) \right)^2   + [a_1(t,s),  a_2(t,s)  ] \right\} \cdots .
\end{array} \]
To deal with the smallness, the same complex number $\kappa$, whose amplitude is taken to be sufficiently large, is introduced as
\[ \begin{array}{ll}
e^{ a_1(t,s) } e^{  a_2(t,s)}  + \kappa I   \vspace{1.5mm} \\
\quad =  (\kappa +1) I + ( a_1(t,s) + a_2(t,s))
 + \frac{1}{2} \left\{ \left( a_1(t,s)   + \frac{1}{2} a_2(t,s) \right)^2   + [a_1(t,s),  a_2(t,s) ] \right\} \cdots  \vspace{1.5mm} \\
\quad = ( \kappa + 1 ) \left[  I + (\kappa +1)^{-1} \left\{ ( a_1(t,s) + a_2(t,s))
 + \frac{1}{2} \left( \left( a_1(t,s)   +  a_2(t,s) \right)^2   + [a_1(t,s),  a_2(t,s) ] \right) \cdots \right\}  \right].
\end{array} \]
By dividing by nonzero $\kappa+1$,
\[ \begin{array}{ll}
( \kappa +1)^{-1} e^{ a_1(t,s) } e^{  a_2(t,s)}  + \kappa (\kappa + 1)^{-1} I   \vspace{1.5mm} \\
 =   I + (\kappa +1 )^{-1} \left\{ ( a_1(t,s) + a_2(t,s))
 + \frac{1}{2} \left( \left( a_1(t,s)   +  a_2(t,s) \right)^2   + [a_1(t,s),  a_2(t,s) ] \right) \cdots \right\} , 
\end{array} \]
where, by taking sufficient large $\kappa$, the amplitude of $( \kappa +1)^{-1} \{ ( a_1(t,s) + a_2(t,s)) \cdots \}$ is small enough to be less than 1.
It leads to
\[ \begin{array}{ll}
  {\rm Log}  \left(  e^{ a_1(t,s) } e^{  a_2(t,s)}  +  \kappa  I  \right)  \vspace{1.5mm} \\
 \quad =     {\rm Log} (\kappa + 1) +
  (\kappa +1 )^{-1} ( a_1(t,s) +  a_2(t,s) ) + \frac{1}{2} (\kappa +1 )^{-1}  [ a_1(t,s), a_2(t,s)]  + \cdots  ,
\end{array} \]
where the left hand side corresponds to the alternative infinitesimal generator of $ e^{ a_1(t,s) } e^{  a_2(t,s)} $ (cf. Eq.~\eqref{altn}).
A similar equality to the original formula shown in Lemma~\ref{lem2} is obtained by taking $\kappa \to 0$. 
By defining such alternative infinitesimal generators ${\hat \alpha}_1(t,s)$ and ${\hat \alpha}_2(t,s)$ by
\[ \begin{array}{ll}
e^{{\hat \alpha}_1(t,s)} e^{{\hat \alpha}_2(t,s)} 
=  e^{ a_1(t,s) } e^{  a_2(t,s)} 
 +  \kappa  I  ,
\end{array} \]
the formula is obtained as
\[ \begin{array}{ll}
e^{{\hat \alpha}_1(t,s)} e^{{\hat \alpha}_2(t,s)}  
 =   e^{   {\rm Log} (\kappa + 1) + (\kappa +1 )^{-1} ( a_1(t,s) +  a_2(t,s) ) + \frac{1}{2} (\kappa +1 )^{-1}  [ a_1(t,s), a_2(t,s)]  + \cdots } .
\end{array} \]
where it is not necessary to take $t=1$ (cf. proof of Lemma 2). and not necessary to assume the smallness of the regularized evolution operator $ e^{a_i(t,s)} $.
\end{proof}

Since the formal solution of first-order linear ordinary differential equation is written as $  e^{\int A(t) dt}$, the next theorem makes sense in terms of taking into account $t$-dependence of infinitesimal generator.
Furthermore a complex constant $\kappa$ does not show up in the next formula. 

\begin{theorem} [Generalized Baker-Campbell-Hausdorff formula: Type II]  \label{cor3}
Let $U_1(t,s)$ and $U_2(t,s)$ be two-parameter semigroups of operator (in other words, evolution operators) in a Banach space $X$.
The evolution operators $U_1(t,s)$ and $U_2(t,s)$ are assumed to be generated by generally unbounded operators $A_1(t)$ and $A_2(t)$ in $X$, respectively. 
Let operators $\partial_t a_1(t,s)$ and  $\partial_t a_2(t,s)$ be alternative infinitesimal generators for $A_1(t)$ and $A_2(t)$ respectively.
The alternative infinitesimal generator $\partial_t a_i(t,s)$ is assumed to commute with both $e^{\int a_1(t,s) dt}$ and $e^{\int a_2(t,s) ds}$.
The formula of exponential function of operators
\begin{equation} \begin{array}{ll}
  e^{\int a_1(t,s) dt} e^{ \int a_2(t,s) dt}
 =  \int a_1(t,s) dt |_{t=0}  + \int a_2(t,s) dt |_{t=0}    \vspace{1.5mm} \\
\quad  +  e^{   ( a_1(0,s) + a_2(0,s) )t + \frac{1}{2}   [ a_1(0,s) , a_2(0,s) ]    t^2   + \cdots }  \vspace{2.5mm} \\
\end{array} \end{equation}
and
\begin{equation} \begin{array}{ll}
   e^{\int a_1(t,s) ds} e^{ \int a_2(t,s) ds}
 =   \int a_1(t,s) ds |_{s=0}  + \int a_2(t,s) ds |_{s=0}   \vspace{1.5mm} \\
\quad  + e^{   ( a_1(t,0) + a_2(t,0) )s + \frac{1}{2}   [ a_1(t,0) , a_2(t,0) ]    s^2   + \cdots } 
\end{array} \end{equation}
are valid in $X$ by assuming a sufficiently small time interval $t,s \in [0,T]$.
For the definition of $ \int a_i(t,s) dt |_{t=0} $ or $\int a_i(t,s) ds |_{s=0} $, $t=0$ or $s=0$ is substituted after calculating an indefinite integral $ \int a_i(t,s) dt $ or $ \int a_i(t,s) ds$.
\end{theorem}

\begin{proof}
The proof is made similar to Lemma \ref{lem2}.
The first derivative is 
\[ \begin{array}{ll}
\frac{d}{dt} \left( e^{\int a_1(t,s) dt} e^{ \int a_2(t,s) dt} \right)
 =  e^{\int a_1(t,s) dt}  ( a_1(t,s) +  a_2(t,s) )  e^{\int a_2(t,s) dt} ,  
\end{array} \] 
and the second derivative is
\[ \begin{array}{ll}
 \frac{d^2}{dt^2}  \left( e^{\int a_1(t,s) dt} e^{ \int a_2(t,s) dt} \right)
 = \frac{d}{dt} \left(  e^{\int a_1(t,s) dt}  ( a_1(t,s) +  a_2(t,s)  e^{\int a_2(t,s) dt}  \right) \vspace{1.5mm} \\
 =   e^{\int a_1(t,s) dt} \left\{ a_1(t,s)  ( a_1(t,s) + a_2(t,s) )  +  ( a_1(t,s) +  a_2(t,s) ) a_2(t,s) \right\}  e^{\int a_2(t,s) dt}  \vspace{1.5mm} \\
 = e^{\int a_1(t,s) dt}  \left\{   ( a_1(t,s) + a_2(t,s) )^2 + [ a_1(t,s),  a_2(t,s) ] \right\}   e^{\int a_2(t,s) dt}  
\end{array} \]
resulting in
\[ \begin{array}{ll}
 e^{- \int a_1(t,s) dt} |_{t=0} ~ \left( e^{\int a_1(t,s) dt} e^{ \int a_2(t,s) dt} \right) ~  e^{-\int a_2(t,s) dt} |_{t=0}  \vspace{1.5mm} \\
 =  I + ( a_1(0,s) +  a_2(0,s) )t 
  +  \frac{1}{2} \left( (  a_1(0,s) +  a_2(0,s) )^2 + [  a_1(0,s), a_2(0,s) ]  \right)t^2 + \cdots .
\end{array} \]
Note again that, as seen in the convergence radius arguments in infinite series, the asymptotic expansion itself cannot be valid for unbounded $A$ and $B$ in $X$.

Next, its logarithm being defined by the Liesz-Dunford integral becomes
\[ \begin{array}{ll}
{\rm Log} \left(   e^{- \int a_1(t,s) dt} |_{t=0} ~ \left( e^{\int a_1(t,s) dt} e^{ \int a_2(t,s) dt} \right) ~  e^{-\int a_2(t,s) dt} |_{t=0}  \right) \vspace{1.5mm} \\
 = {\rm Log} \left(  I + ( a_1(0,s) +  a_2(0,s) )t 
 +  \frac{1}{2} \left( (  a_1(0,s) +  a_2(0,s) )^2 + [  a_1(0,s), a_2(0,s) ]  \right)t^2 + \cdots  \right)    \vspace{1.5mm} \\
 =   ( a_1(0,s) + a_2(0,s) )t + \frac{1}{2} \left( ( a_1(0,s)  +  a_2(0,s) )^2   +  [ a_1(0,s) , a_2(0,s) ]  \right)t^2 + \cdots)   \vspace{0.5mm} \\
\qquad  - \frac{1}{2} \left\{  ( a_1(0,s) + a_2(0,s) )t + \frac{1}{2} \left( ( a_1(0,s)  +  a_2(0,s) )^2   +  [ a_1(0,s) , a_2(0,s) ]  \right)t^2 + \cdots) \right\} ^2  \vspace{0.5mm} \\
\qquad  + \frac{1}{3} \left\{  ( a_1(0,s) + a_2(0,s) )t + \frac{1}{2} \left( ( a_1(0,s)  +  a_2(0,s) )^2   +  [ a_1(0,s) , a_2(0,s) ]  \right)t^2 + \cdots) \right\} ^3   \vspace{1.5mm} \\
 =   ( a_1(0,s) + a_2(0,s) )t + \frac{1}{2}   [ a_1(0,s) , a_2(0,s) ]    t^2   + \cdots
\end{array} \]
if $ \left\| ( a_1(0,s) +  a_2(0,s) )t  +  \frac{1}{2} \left( (  a_1(0,s) +  a_2(0,s) )^2 + [  a_1(0,s), a_2(0,s) ]  \right)t^2 + \cdots \right\| < 1 $ is satisfied.
Unlike Lemma \ref{lem2}, this condition is always satisfied by the validity of continuity assumption for $U_i(t,s)$ (i.e. $C_0$-semigroup) and therefore for $a_i(t,s)$.
It is notable here that a setting $t=1$ cannot make sense because $a_i(t,s)$ are $t$-dependent.
Here is the difference of condition compared to Lemma \ref{lem2}.
The second equality is obtained by the same way as the first equality.
\end{proof}

\subsection{Unbounded formulation of von Neumann equations}
A general form of the von Neumann, which is general enough to include differential operators and the other unbounded operators, is proposed here.
Let the evolution operator $U_i(t,s)$ in a Banach space $X$ be generated by $A_i(t): D(A_i(t)) \to X$ for $i=1,2$, and $a_1(t,s)$ and  $a_2(t,s)$ be alternative infinitesimal generators of $A_1(t)$ and $A_2(t)$ respectively.
In particular, $A_1(t)$ and/or $A_2(t)$ are assumed to be unbounded operator in $X$.
In this case usual type of the Campbell-Baker-Hausdorff formulae shown in Lemmas~\ref{lem1} and \ref{lem2} are not valid, and it is necessary to replace them by the formulae obtained in Theorems~\ref{thm1} to \ref{cor3}.
This fact affects the valid formulation of von Neumann type equations and therefore Liouville type equations.
Here the fundamental formalism of von Neumann equation, which is valid to be the commutation of unbounded operators, is presented.
 
\begin{theorem}[Generalized von Neumann equation] \label{cor4}
Let $U_1(t,s)$ and $U_2(t,s)$ be two-parameter semigroups of operator (in other words, evolution operators) in a Banach space $X$.
The evolution operators $U_1(t,s)$ and $U_2(t,s)$ are assumed to be generated by generally unbounded operators $A_1(t)$ and $A_2(t)$ in $X$, respectively. 
Let operators $\partial_t a_1(t,s)$ and  $\partial_t a_2(t,s)$ be alternative infinitesimal generators for $A_1(t)$ and $A_2(t)$ respectively.
The alternative infinitesimal generator $\partial_t a_i(t,s)$ is assumed to commute with both $e^{\int a_i(t,s) dt}$ and $e^{\int a_i(t,s) ds}$ ($i=1,2$).
The unbounded operator version of von Neumann equation is written by the logarithm of operator
\begin{equation} \begin{array}{ll}
\frac{\partial \hat{\rho} (t,0)}{\partial t}
 = \frac{i}{\hbar} [\hat{\rho}(t,0),\hat{H}(t)]  
 = \left. \frac{i}{\hbar} \partial_s^2 {\rm Log} \left( e^{\int \hat{\rho}(t,s) ds}  e^{\int \hat{H}(t) ds} \right) \right|_{s=0},
\end{array} \end{equation}
iwhere ${\hat H}$ is the alternative infinitesimal generator of Hamiltonian operator $H$, $\hat{\rho}(t,s)$ denotes the alternative infinitesimal generator of two-parameter density operator $\rho(t,s)$, and the logarithm is defined by the Riesz-Dunford integral.
Also in this generalized von Neumann equation, the differential is defined by a kind of weak differential \cite{20iwata-4}..
\end{theorem} 
 
\begin{proof}
The validity of the following two formulae is understood in Theorem \ref{cor3}.
\[ \begin{array}{ll}
{\rm Log} \left( e^{\int a_1(t,s) ds } e^{\int a_2(t,s) ds} \right)
 =  \int a_1(t,s) ds |_{s=0}  + \int a_2(t,s) ds |_{s=0}   \vspace{1.5mm} \\
\quad  +    ( a_1(t,0) + a_2(t,0) )s + \frac{1}{2}   [ a_1(t,0) , a_2(t,0) ]    s^2   + \cdots  \vspace{2.5mm} \\
{\rm Log} \left( e^{\int a_1(t,s) ds} e^{- \int a_2(t,s) ds} \right)
=  \int a_1(t,s) ds |_{s=0}  - \int a_2(t,s) ds |_{s=0}   \vspace{1.5mm} \\
\quad  +    ( a_1(t,0) - a_2(t,0) )s - \frac{1}{2}   [ a_1(t,0) , a_2(t,0) ]    s^2   + \cdots  
\end{array} \]
leading to
\[ \begin{array}{ll}
\partial_s {\rm Log} \left( e^{\int a_1(t,s) ds } e^{\int a_2(t,s) ds} \right)) |_{s=0}
 =  a_1(t,0) +a_2(t,0),
 \vspace{2.5mm} \\
\partial_s^2 {\rm Log} \left( e^{\int a_1(t,s) ds } e^{\int a_2(t,s) ds} \right) |_{s=0}
 =  [a_1(t,0),a_2(t,0)]   ,
  \vspace{2.5mm} \\
\partial_s  {\rm Log}  \left( e^{\int a_1(t,s) ds} e^{- \int a_2(t,s) ds} \right) |_{s=0}
 = a_1(t,0) - a_2(t,0)  ,
 \vspace{2.5mm} \\
\partial_s^2  {\rm Log}  \left( e^{\int a_1(t,s) ds} e^{- \int a_2(t,s) ds} \right) |_{s=0}
 = -  [a_1(t,0),a_2(t,0)]   = [a_2(t,0),a_1(t,0)] ,
 \vspace{2.5mm} \\
 \partial_s^2  {\rm Log}  \left( e^{\int a_2(t,s) ds} e^{- \int a_1(t,s) ds} \right) |_{s=0}
 = -  [a_2(t,0),a_1(t,0)]   = [a_1(t,0),a_2(t,0)] ,
\end{array} \]
where note that $e^{a_i(t,0)}$ necessarily exist, because of the boundedness of $a_i (t,0)$ on $X$.
They result in
\[ \begin{array}{ll}
    [a_1(t,0),a_2(t,0)]   =
\left. \partial_s^2 {\rm Log} \left( e^{\int a_1(t,s) ds } e^{\int a_2(t,s) ds} \right) \right|_{s=0}   \vspace{1.5mm} \\
=
\left. \partial_s^2  {\rm Log}  \left( e^{ \int a_2(t,s) ds} e^{- \int a_1(t,s) ds} \right) \right|_{s=0}
\end{array} \]
Then the following identity is found.
\begin{equation} \begin{array}{ll}
\left. \partial_s^2  {\rm Log} \left( e^{\int a_1(t,s) ds } e^{\int a_2(t,s) ds} \right) \right|_{s=0}  
- \left.  \partial_s^2  {\rm Log}  \left( e^{  \int a_2(t,s) ds} e^{- \int a_1(t,s) ds} \right) \right|_{s=0}  = 0
\end{array} \end{equation}
where the logarithm is defined by the Riesz-Dunford integral, and the differential is defined by a kind of weak differential \cite{20iwata-4}.
The identity shows the commutation relation of evolution operators, the logarithm, and the second-order differential.

Let $\hat{H}(t)$ be the alternative infinitesimal generator of Hamiltonian operator of a certain physical system.
Here the Neumann type equation (corresponding to the quantum version the Liouville type equation), which is equivalent to the equation of motion, is written by
\[
 \frac{\partial \hat{\rho}(t,0)}{\partial t} = \frac{i}{\hbar} [\hat{\rho}(t,0) ,\hat{H}(t)]
\]
for a certain operator meaning a certain physical quantity, where $\hat{\rho}(t,0)$ is the alternative infinitesimal generator of $t$-dependent operator meaning the density (often called density matrix in physics).
According to
\[ \begin{array}{ll}
\left. \partial_s^2  {\rm Log}  \left( e^{\int a_1(t,s) ds} e^{ \int a_2(t,s) ds} \right) \right|_{s=0}
 =  [a_1(t,0),a_2(t,0)]  
\end{array} \]
and setting one-parameter operator $a_2(t,0) = \hat{H}(t)$ assuming the alternative infinitesimal generator of $t$-dependent Hamiltonian operator,
\[ 
\frac{\partial \hat{\rho}(t,0)}{\partial t} = \left. \frac{i}{\hbar} \partial_s^2 {\rm Log} \left( e^{\int \hat{\rho}(t,s) ds}  e^{\int \hat{H}(t) ds} \right) \right|_{s=0}
 \]
is obtained, where $\hat{\rho}(t,s)$ denotes two-parameter density operator.
\end{proof}

This theorem shows how the second derivative of logarithm is associated with the natural laws or more concretely the equations of motion. 
It is worth reminding here that the Toda lattice equation \cite{69toda, 7073hirota} which includes soliton solutions, is written by the differential of logarithm.

\section{Summary}
By means of the logarithmic representation of infinitesimal generators, 
\begin{itemize}
\item Campbell-Baker-Hausdorff formulae (Theorems~\ref{thm1} to \ref{cor3})  \vspace{1.5mm}
\item von Neumann equation (Theorem~\ref{cor4})
\end{itemize}
have been generalized to the cases including the commutation of unbounded operators, where the concept of alternative infinitesimal generators and the regularized evolution operators \cite{17iwata-3} are effectively utilized.
It is obvious that the second-order derivative of logarithm is an alternative concept for the commutator product.
In other words, Theorem~\ref{cor4} is a correspondence theorem between the commutator product and the second order derivative of logarithm.
Logarithmic representation is expected to be analytically advantageous in some cases compared to the algebraic commutator representation.
 
Since it is necessary to consider differential operators in describing physical laws, the results presented in Theorems~\ref{thm1} to \ref{cor3} are expected to have an impact from an applicational point of view.
Since the commutation is an important ingredient of modern physics regardless of the size of physical systems, the result presented in Theorem~\ref{cor4} shows that the second-order differential of logarithm definitely involve the concept of commutation.
It is notable that the derivative of logarithm has been known to appear in the master equation of soliton equations \cite{81ablowitz}.
For a previous work in the similar direction with respect to the Lie algebra, see unbounded generalization of rotation group \cite{18iwata-2}.\vspace{6mm}  \\

{\bf Conflict of interest.} 
The corresponding author states that there is no conflict of interest. \vspace{6mm} \\

{\bf Data availability.} 
The current study has no associated data. \vspace{6mm} \\


\begin{thebibliography}{30}
 \bibitem{81ablowitz}
 M.  J. Ablowitz  and  H.  Segur,
Solitons  and  the inverse  scattering  transform,
Philadelphia, SIAM, 1981.

\bibitem{51cole}
J. D. Cole, 
On a quasi-linear parabolic equation occurring in aerodynamics, 
Quart. Appl. Math. 9 (1951) no. 3, 225-236.

\bibitem{43dunford}
N. Dunford, 
Spectral theory. I. Convergence to projections, 
Trans. Amer. Math. Soc. 54 (1943) 185-217.001.

\bibitem{50hopf}
E. Hopf,
The partial differential equation $u_{t} + uu_{x} = \mu u_{xx}$, 
Comm. Pure Appl. Math. 3 (1950) 201-230.

\bibitem{7073hirota}
R. Hirota and K. Suzuki,
Studies on Lattice Solitons by Using Electrical Circuit, 
J. Phys.Soc. Jpn. 28(1970) 1366-1367; 
Theoretical and Experimental Studies of Lattice Solitons in Nonlinear Lumped Networks, 
Proc. IEEE 61(1973)
1483-1491

\bibitem{17iwata-1}
Y. Iwata,
Infinitesimal generators of invertible evolution families,
Methods Funct. Anal. Topology {\bf 23} 1 (2017) 26-36. 

\bibitem{17iwata-3}
Y. Iwata,
Alternative infinitesimal generator of invertible evolution families, 
J. Appl. Math. Phys. {\bf 5} (2017) 822-830. 

\bibitem{20iwata-4}
Y. Iwata,
Operator topology for logarithmic infinitesimal generators.
A chapter of a book "Structural topology and symplectic geometry", IntechOpen, 2020.

  \bibitem{18iwata-2} 
Y. Iwata, 
Unbounded formulation of the rotation group,
J. Phys: Conf. Ser. 1194 (2019) 012053. 


\bibitem{19iwata-2}
Y. Iwata,
Abstract formulation of the Cole-Hopf transform,
Methods Funct. Anal. Topology {\bf 25} 2 (2019) 142-151. 

\bibitem{20iwata-1}
Y. Iwata,
Abstract formulation of the Miura transform, 
Mathematics 8 (2020) 747.

\bibitem{18iwata}
Y. Iwata,
Operator algebra as an application of logarithmic representation of infinitesimal generators
J. Phys: Conf. Ser. 965 (2018) 012022. 

\bibitem{20iwata-2}
Y. Iwata,
Theory of $B(X)$-module: algebraic module structure of generally-unbounded infinitesimal generators,
Adv. Math. Phys. Vol. 2020, Article ID 3989572 (2020). 

  \bibitem{21iwata} 
Y. Iwata, 
Unbounded generalization of logarithmic representation of infinitesimal generators,
arXiv:2101.01443.

\bibitem{61kato}
T. Kato,
Abstract evolution equations of parabolic type in Banach and Hilbort spaces, 
Nagoya Math. J. 19 (1961) 93-125.

\bibitem{70kato}
T. Kato, 
Linear evolution equations of hyperbolic-type,
J. Fac. Sci. Univ. Tokyo Sect. I 17 (1970) 241-258.

\bibitem{73kato}
T. Kato, 
Linear evolution equations of hyperbolic type. II,
J. Math. Soc. Japan 25 (1973) 648-666.

\bibitem{68miura}
R. Miura,
Korteweg-de Vries Equation and Generalizations I. a remarkable explicit nonlinear transformation, 
J. Math. Phys. 9 (1968) 1202-1204.

 \bibitem{73sagel}
A. A. Sagle and R. E. Walde, 
Introduction to Lie groups and Lie algebras, 
Academic press, inc, London 1973.

 \bibitem{60tanabe}
H. Tanabe, 
On the equations of evolution in a Banach space, 
Osaka Math. J., 12 (1960) 363-376.

 \bibitem{61tanabe}
H. Tanabe, 
Evolution equations of parabolic type, Proc. Japan Acad., 37, 10 (1961) 610-613.

 \bibitem{79tanabe}
H. Tanabe, Equations of evolution. Pitman, 1979. 

\bibitem{69toda}
M. Toda,
Mechanics and Statistical Mechanics of Nonlinear Chains, 
J. Phys. Soc. Jpn. Suppl. 26(1969) 235-237.


\end{thebibliography}
\end{document}